\documentclass[a4paper,reqno,11pt,oneside]{amsart}
\usepackage{amsmath,geometry}
\usepackage{graphicx}
\usepackage{mathrsfs}
\usepackage{enumerate}
\usepackage{amssymb,amsmath, amsfonts, amsthm}
\usepackage{latexsym}
\usepackage{color} 
\usepackage[dvips]{epsfig}
\usepackage[colorlinks]{hyperref}
\usepackage{bm}
\renewcommand{\(}{\left(}
\renewcommand{\)}{\right)}

\DeclareMathOperator{\ind}{ind}

\theoremstyle{plain}
\newtheorem{definition}{Definition}[section]
\newtheorem{theorem}[definition]{Theorem}
\newtheorem{proposition}[definition]{Proposition}

\newtheorem{corollary}[definition]{Corollary}
\newtheorem{remark}[definition]{Remark}

\begin{document}
\title[]{On the critical points of Steklov eigenfunctions}
\thanks{The first author is partially supported by MUR-PRIN-2022AKNSE4 ``Variational and Analytical aspects of Geometric PDEs''. The second author is partially supported by  the MUR-PRIN-20227HX33Z ``Pattern formation in nonlinear phenomena''. The first and the second authors are also partially supported by INDAM-GNAMPA project ``Problemi di doppia curvatura su varietà a bordo e legami con le EDP di tipo ellittico''. The third author acknowledges support of the INDAM-GNSAGA project ``Analisi Geometrica: Equazioni alle Derivate Parziali e Teoria delle Sottovarietà'' and of the the project ``Perturbation
problems and asymptotics for elliptic differential equations: variational and potential theoretic methods'' funded by the MUR Progetti di Ricerca di Rilevante Interesse Nazionale (PRIN) Bando 2022 grant 2022SENJZ3.}
\author[Battaglia]{Luca Battaglia}
\address{Dipartimento di Matematica e Fisica, Universit\`a degli Studi Roma Tre, Via della Vasca Navale 84 - 00146 Roma, Italy, e-mail: {\sf luca.battaglia@uniroma3.it}}
\author[Pistoia]{Angela Pistoia}
\address{Dipartimento di Scienze di Base e Applicate per l'Ingegneria, Universit\`a di Roma ``La Sapienza'', Via Scarpa 12 - 00161 Roma, Italy, e-mail: {\sf angela.pistoia@uniroma1.it}.}
\author[Provenzano]{Luigi Provenzano}
\address{Dipartimento di Scienze di Base e Applicate per l'Ingegneria, Universit\`a di Roma ``La Sapienza'', Via Scarpa 12 - 00161 Roma, Italy, e-mail: {\sf luigi.provenzano@uniroma1.it}.}

\begin{abstract}
We consider the critical points of Steklov eigenfunctions on a compact, smooth $n$-dimensional Riemannian manifold $M$ with boundary $\partial M$. For generic metrics on $M$ we establish an identity which relates the sum of the indexes of a Steklov eigenfunction, the sum of the indexes of its restriction to $\partial M$, and the Euler characteristic of $M$. In dimension $2$ this identity gives a precise count of the interior critical points of a Steklov eigenfunction in terms of the Euler characteristic of $M$ and of the number of sign changes of $u$ on $\partial M$. In the case of the second Steklov eigenfunction on a genus $0$ surface, the identity holds for any metric. As a by-product of the main result, we show that for generic metrics on $M$ Steklov eigenfunctions are Morse functions in $M$.
\end{abstract}



\maketitle

\section{Introduction}

Let $(M,g)$ be a smooth compact $n$-dimensional Riemannian manifold with non-empty, smooth boundary $\partial M$. We consider the Steklov eigenvalue problem, namely
\begin{equation}\label{Steklov}
\begin{cases}
\Delta u=0\,, & {\rm in\ } M\\
\partial_{\nu}u=\sigma u\,, & {\rm on\ }\partial M.
\end{cases}
\end{equation}
Here $\nu$ denotes the outer unit normal to $\partial M$ and $\partial_{\nu}u:=\langle\nabla u,\nu\rangle$ is the derivative of $u$ in the direction of $\nu$ along $\partial M$. This problem was introduced in \cite{steklov} at the beginning of the 20\textsuperscript{th} century, and has been intensively studied since then. Provided that the trace operator from $H^1(M)$ to $L^2(\partial M)$ is compact (as is the case of a smooth compact manifold with smooth boundary), problem \eqref{Steklov} admits an increasing sequence of non-negative eigenvalues of finite multiplicity
$$
0=\sigma_1<\sigma_2\leq\cdots\leq\sigma_k\leq\cdots\nearrow+\infty.
$$
The corresponding eigenfunctions are denoted by $u_1,u_2,\cdots$. In particular, $u_1$ is a constant function, and the eigenfunctions $\{u_k\}_{k=1}^{\infty}$ can be chosen such that their traces form an orthonormal basis of $L^2(\partial M)$. We also recall that the Steklov eigenvalues $\sigma_k$ can be interpreted as the eigenvalues of the Dirichlet-to-Neumann map $\mathscr D:H^{1/2}(\partial M)\to H^{-1/2}(\partial M)$, which acts as follows: $\mathscr D f=\partial_{\nu}(Hf)$, where $Hf$ is the harmonic extension of $f$ in $M$. It turns out that the eigenfunctions of $\mathscr D$ are the traces of the Steklov eigenfunctions $u_k$. Interested readers may find quite complete information on the Steklov problem in the survey \cite{GP_survey}, and in the more recent \cite{CGGS_survey}, along with a large number of open questions.

\medskip

In this paper we are interested in the geometry of the Steklov eigenfunctions, and in particular in the counting of the interior critical points. The literature on the geometry of the Steklov eigenfunctions has been mainly devoted to questions like the number and measure of nodal domains. This is not a surprise, since the study of nodal domains is perhaps one of the oldest topics in spectral geometry. In particular, the Courant's nodal domain theorem holds for Steklov eigenfunctions: $u_k$ has at most $k$ nodal domains (see e.g., \cite{KKGP_steklov}). However, the proof of Courant's nodal domain theorem cannot be generalized to the Dirichlet-to-Neumann eigenfunctions, i.e., the traces of the Steklov eigenfunctions on $\partial M$, and in fact it is an open problem to find an upper bound for the nodal count of Dirichlet-to-Neumann eigenfunctions. Bounds are available only in two dimensions as consequence of the interior nodal count and elementary topological arguments. For simply connected domains, see e.g., \cite{AM_steklov}. In dimension two, bounds on the number of nodal domains can be translated into bounds on the multiplicities of $\sigma_k$. The study of multiplicity bounds is another topic which has been extensively investigated in the last decades (see e.g., \cite{jammes_steklov,KKGP_steklov}, see also \cite[\S 6]{GP_survey} and \cite[\S 10]{CGGS_survey} for a more up-to-date account on multiplicity bounds and related open problems). Concerning critical points of eigenfunctions, much less is known. Most of the results are for simply connected domains in $\mathbb R^2$. The main reference is \cite{AM_steklov}, where the authors prove a number of bounds concerning the number of interior and boundary nodal domains, the eigenvalue multiplicity, and number of interior critical points of a Steklov eigenfunction. We also refer to \cite{AM_1992} where the authors consider more in general identities or inequalities relating the number and type of critical points of solutions to elliptic PDEs on domains of the plane, the topology of the domain, and the boundary data. We also mention \cite{PP} where the authors consider an overdetermined Steklov problem and prove, as a technical lemma, that for a simply connected planar domain, the second Steklov eigenfunction has no critical points. The proof can be adapted to any simply connected surface. Up to our knowledge, these are the only results on the counting of critical points of Steklov eigenfunctions.

\medskip

In our paper, we compute the number of interior critical points of $u_k$ in terms of the number of critical points of $u_k$ on $\partial M$ (or, to be more precise, the number of critical points of the corresponding Dirichlet-to-Neumann eigenfunction on $\partial M$). One of the main results is valid for {\it generic} metrics on a given smooth manifold $M$, and can be summarized as follows: for a generic metric, any Steklov eigenfunction $u$ satisfies
\begin{equation}\label{main1}
\sum_{\{x\in M:\nabla u(x)=0\}}{\ind_x}\nabla u=\chi(M)-\sum_{\{\xi\in \partial M:\bar\nabla u(\xi)=0,u(\xi)<0\}}{\ind_{\xi}}\bar\nabla u,
\end{equation}
where $\chi(M)$ is the Euler characteristic of $M$.
Here $\bar\nabla$ denotes the gradient on $\partial M$ with respect to the induced metric. By an abuse of notation, in the right-hand side of \eqref{main1} we have still denoted by $u$ the trace of $u$ (which in the smooth case is just the restriction of $u$ to $\partial M$). With the term {\it generic} we intend that the result holds for a dense subset of Riemannian metrics on the given manifold $M$. As a consequence, we show that \eqref{main1} implies 
\begin{equation}\label{main2}
\sum_{\{x\in M:\nabla u(x)=0\}}{\ind_x}\nabla u=\chi(M)-\sum_{i=1}^{\ell}\chi(\mathscr P_i).
\end{equation}
where $\mathscr P_i$ are the connected components of $\mathscr P=\{\xi\in\partial M:u(\xi)<0\}$.

\medskip

In order to prove \eqref{main1} we exploit the fact that, generically, a Steklov eigenfunction is Morse in $M$, its trace is Morse on $\partial M$ (i.e., Dirichlet-to-Neumann eigenfunctions are Morse functions), and its restriction on the boundary has no singular zeros (i.e., Dirichlet-to-Neumann eigenfunctions have no singular zeros). The last two generic statements have been recently proved in \cite{wang}. In fact, in \cite{wang} the author proves generic properties of Dirichlet-to-Neumann eigenfunctions. We prove in this paper that Steklov eigenfunctions are generically Morse in $M$ and, for completeness, we show also that, generically, there are no singular zeros in $M$.

\medskip

When $n=2$ (i.e., for Riemannian surfaces), formula \eqref{main2} implies that

\begin{equation}\label{main3}
\sharp\{{\rm critical\ points\ of\ }u{\rm\ in\ }M\}=\sum_{j=1}^L\ell_j-\chi(M),
\end{equation}
where $L$ is the number of connected components of $\partial M$ where $u$ changes sign exactly $\ell_j$ times on the $j$-th component. We remark that only the connected components of the boundary where $u$ changes sign influence the number of critical points of $u$. 

Finally, we prove that for surfaces of genus $0$, and for the second eigenfunction $u_2$, formula \eqref{main3} holds for any metric (not just generically). This is a consequence of the fact that the second eigenfunction is always a Morse function, along with other properties which we prove in the paper.

\medskip

We remark that for simply connected surfaces, formula \eqref{main3} implies that there are no interior critical points. This fact is known for simply connected planar domains by \cite{PP}, however the proof in \cite{PP} readily extends to the case of simply connected surfaces.

\medskip

We point out that the same results can be proved, using essentially the same arguments, for the \emph{weighted} Steklov problem
$$\begin{cases}
\Delta u=0\, & {\rm in\ } M\\
\partial_{\nu}u=\sigma\rho u\, & {\rm on\ }\partial M,
\end{cases}$$
where $\rho=\rho(x)>0$ is a strictly positive smooth function.\\
It would be interesting to investigate the same issues in the case of the \emph{sloshing} problem, namely when $\rho$ can equal zero in some portions of $\partial M$ (see for instance \cite{sloshing}).

\medskip

The present paper is organized as follows: in Section \ref{S0} we prove and identity which relates the sum of the indexes of an harmonic function $u$ on $M$, the sum of the indexes of its restriction on $\partial M$, and $\chi(M)$, which holds under suitable assumptions on $u$ (see Theorem \ref{T_0}). In section \ref{S1} we apply Theorem \ref{T_0} to Steklov eigenfunctions. In fact, for generic metrics, the hypotheses of Theorem \ref{T_0} are satisfied by Steklov eigenfunction. This allows to prove \eqref{main1} (see Theorem \ref{T_2}). In Section \ref{S1} we also obtain \eqref{main2} and \eqref{main3} as consequences of \eqref{main1}. In Section \ref{S2} we observe that \eqref{main3} holds for {\it any} Riemannian surface of genus $0$ when we consider the second Steklov eigenfunction. To prove the result, a crucial observation is that in this case the second Steklov eigenfunction is always a Morse function. Finally, in Appendix \ref{gen} we prove that for a generic metric on $M$, all Steklov eigenfunctions are Morse functions.

\section{A formula for the critical points of harmonic functions}\label{S0}

Let $(M,g)$ be a smooth, compact $n$-dimensional Riemannian manifold with smooth boundary $\partial M$. Through all the paper we shall denote by $\nabla,\Delta, D^2$ the gradient, Laplacian and Hessian on $M$ with respect to the given metric $g$, respectively, and by $\bar\nabla$ the gradient on $\partial M$ with respect to the induced metric. 

\medskip

Given a vector field $V$ on a smooth manifold $M$ with an isolated zero $x$, the {\it index} of $V$ at $x$, or $\ind_x V$, is defined as the local degree of the map $\Psi:\partial U\to\mathbb S^{n-1}$ given by $\Psi:=\frac{V}{|V|}$. Here we have fixed a system of local coordinates around $x$ and $U$ is a small coordinate neighborhood of $x$ which does not contain other zeros of $V$.

\medskip

The aim of this section is to prove the following theorem.

\begin{theorem}\label{T_0}
Let $(M,g)$ be a smooth compact Riemannian manifold with smooth boundary $\partial M$. Let $u$ be a smooth function on $M$ which satisfies $\Delta u=0$ in $M$, and $u_{|_{\partial M}}=h$. Assume moreover that
\begin{enumerate}[i)]
\item $u$ has isolated critical points in $M$;
\item $h$ has isolated critical points on $\partial M$;
\item $\nabla u\ne 0$ on $\partial M$.
\end{enumerate}
Then
\begin{equation}\label{E0}
\sum_{\{x\in M:\nabla u(x)=0\}}{\ind_x}\nabla u=\chi(M)-\sum_{\{\xi\in\partial M:\bar\nabla h(\xi)=0,\partial_{\nu}u(\xi)<0\}}{\ind_{\xi}}\bar\nabla h.
\end{equation}
\end{theorem}

A crucial role in the proof of Theorem \ref{T_0} is played by the Poincaré-Hopf Theorem.

\begin{theorem}[Poincaré-Hopf]\label{PH}
Let $M$ be an orientable $n$-dimensional Riemannian manifold, and let $V$ be a smooth vector field on $M$ with isolated zeros $x_i$. If $M$ has a boundary, assume that $\langle V,\nu\rangle>0$, where $\nu$ is the outer unit normal to $\partial M$. Then
$$
\sum_{i}{\ind_{x_i}}V=\chi(M).
$$
\end{theorem}

\begin{proof}[\bf Proof of Theorem \ref{T_0}]
The proof is divided into two step. In the first step we consider the manifold $M$ and take its {\it double}, which is a closed manifold. Starting from $u$, we then define an auxiliary function $\tilde u$ on the double manifold and we relate its critical points to the critical points of $u$ on $M$ and of $h$ on $\partial M$. In the second step we apply the Poincaré-Hopf Theorem to the function $\tilde u$.

{\em The double manifold and the auxiliary function.} We consider the {\it double} of $M$, $\mathscr DM=M\cup_{\mathrm Id}M$, where ${\mathrm Id}:\partial M\to\partial M$ is the identity map of $\partial M$: it is obtained from $M\cup M$ by identifying each boundary point in one copy of $M$ with the same boundary point in the other. We can give a smooth structure on $\mathscr DM$ using a tubular neighborhood of the boundary (which exists since the boundary is at least $C^2$) to define smooth charts in a neighborhood of the gluing (see e.g., \cite[\S 9]{lee} for more details). We have that $\partial M$ is a smooth submanifold of $\mathscr DM$. For $\rho>0$ small, we denote by $\omega_{\rho}$ a tubular neighborhood of $\partial M$ in $\mathscr DM$, namely $\omega_{\rho}:=\{x\in\mathscr DM:{\rm dist}(x,\partial M)<\rho\}$ ($\rho$ is chosen small enough such that any point in $\omega_{\rho}$ has a unique nearest point on $\partial M$).

\medskip

Consider now the smooth function $u$. We define $\tilde u$ on $\mathscr DM$ such that $\tilde u\equiv u$ on $\mathscr DM\setminus\overline{\omega_{\rho}}$. Let $\nu\in T\mathscr DM$ be a unit vector field normal to the submanifold $\partial M$. We define $\Phi:\partial M\times(-\rho,\rho)\to\omega_{\rho}$ as
\begin{equation}\label{fermi}
\Phi(\xi,t):={\rm Exp}_{\xi}(t\nu(\xi)),
\end{equation}
where $\xi\in\partial M$, $t\in(-\rho,\rho)$ and ${\rm Exp}$ is the exponential map. The coordinates $(\xi,t)$ are usually called {\it Fermi} coordinates. We define, when $t\in[0,\rho)$,
$$
\tilde u(\Phi(\xi,t)):=u(\Phi(\xi,\phi(t)))
$$
for $\phi:[0,\rho]\to[0,\rho]$ that satisfies $\phi'(0)=\phi(0)=0$, $\phi'(\rho)=1$, $\phi(\rho)=\rho$, $\phi'>0$ on $(0,\rho)$. This is achieved e.g. for $\phi(t)=\frac{2t^2}{\rho}-\frac{t^3}{\rho^2}$. When $t\in(-\rho,0]$, we define $\tilde u$ by even reflection in the direction of $\nu$. It is easy to check that $\tilde u$ is $C^1$ on $\mathscr DM$.

\medskip

Let us compute now the gradient of $\tilde u$ in $\omega_{\rho}$. Since $\tilde u$ is even with respect to reflections through $\partial M$, for simplicity we consider the points $x\in\omega_{\rho}$ of the form $x=\Phi(\xi,t)$ for $t\in(0,\rho)$.

 We have that $\nabla \tilde u=\partial_t\tilde u\, \nu+\bar\nabla_t\tilde u$, where $\bar\nabla_t$ denotes the gradient for the induced metric on the parallel to $\partial M$ at distance $\phi(t)$ from $\partial M$. We check that
\begin{equation}\label{grad_fermi}
\partial_t\tilde u(\Phi(\xi,t))=\phi'(t)\partial_{\nu}u(\Phi(\xi,\phi(t)))\,,\ \ \ \bar\nabla_t\tilde u(\Phi(\xi,t))=\bar\nabla\tilde u(\Phi(\xi,t))+O(\phi(t)).
\end{equation}
We still denote here by $\nu$ the natural extension of the normal vector field to $\omega_{\rho}$ (i.e., $\nu$ is a unit vector field normal to all parallels to $\partial M$ in $\omega_{\rho}$).

We look now for critical points of $\tilde u$ in $\omega_{\rho}$. From the assumption $iii)$, we deduce that we can choose $\rho$ small enough such that $\nabla u\ne 0$ in $\omega_{\rho}$. Hence, to have both components of the gradient equal to zero, we deduce from \eqref{grad_fermi} that necessarily $t=0$, hence $\bar\nabla\tilde u(\Phi(\xi,0))=0$, which is just $\bar\nabla h=0$. Therefore we conclude that the critical points of $\tilde u$ are given by the union of the interior critical points of $u$ (on each copy of $M$ in the double manifold) and of the critical points of $h$ on $\partial M$.

{\em The Poincaré-Hopf Theorem.} We apply Theorem \ref{PH} to the vector field $\nabla\tilde u$ on the closed manifold $\mathscr DM$. Since by $i)$ and $ii)$ both $u$ and $h$ have isolated critical points (on $M$ and $\partial M$, respectively), the function $\tilde u$ also has isolated critical points on $\mathscr DM$ and so 
\begin{equation}\label{s1}
\begin{aligned} 
\chi\(\mathscr D M\)=&\sum_{\{x \in\omega_{\rho}:\nabla\tilde u(x )=0\}} {\ind_x \nabla\tilde u}+\sum_{\{x \in\mathscr DM\setminus\overline{\omega}_{\rho}\,:\nabla\tilde u(x )=0\}} {\ind_x \nabla\tilde u}\\
=&\sum_{\{x \in\omega_{\rho}:\nabla\tilde u(x )=0\}} {\ind_x \nabla\tilde u}+2\sum_{\{x \in M:\nabla u(x )=0\}} {\ind_x \nabla u}.
\end{aligned} 
\end{equation}

We have seen that the critical points in $\omega_{\rho}$ (if $\rho$ is chosen sufficiently small) actually belong to the submanifold $\partial M$ and they are exactly the same critical points of $h$. However, their index may change, since each can be either a minimum or a maximum in the normal direction to $\partial M$. Let $\xi\in\partial M$ be a critical point for $h$. Note that the sign of the second normal derivative of $\tilde u$ at a critical point $x=(\xi,0)$ is exactly that of $-\partial_{\nu}u(\xi,0)$, hence

$$\begin{aligned}
 \partial_\nu u(\xi,0 )>0\ &\Rightarrow\ 
\ind_x \nabla \tilde u= -\ind_\xi \bar\nabla h\\
 \partial_\nu u(\xi,0 ) <0\ &\Rightarrow\ 
\ind_x \nabla \tilde u= \ind_\xi \bar\nabla h\\
\end{aligned}.$$
To simplify the notation, from now on we shall denote a point of Fermi coordinates $(\xi,0)$ just by $\xi$ (these are the points on $\partial M\subset\mathscr DM$). We have then
\begin{equation}\label{s2}\begin{aligned} 
\sum_{\{x \in\omega_{\rho}:\nabla\tilde u(x )=0\}} {\ind_x \nabla\tilde u}
=&\sum_{\{\xi \in\partial M:\bar\nabla h(\xi )=0,\, \partial_\nu u(\xi)<0\}} {\ind_\xi \bar\nabla h}\\ &-\sum_{\{\xi \in\partial M:\bar\nabla h(\xi )=0,\, \partial_\nu u(\xi)>0\}} {\ind_\xi \bar\nabla h}.
\end{aligned} \end{equation}

On the other hand, applying again the Poincaré-Hopf Theorem, this time to $\bar\nabla h$ on $\partial M$, and taking into account that
for any $\xi $ with $\bar\nabla h(\xi )=0$ it holds true that $\partial_\nu u(\xi)\not=0$ by $iii)$, we get
\begin{equation}\label{s3}\begin{aligned} 
 \chi\(\partial M\) =& \sum_{\{\xi \in\partial M:\bar\nabla h(\xi )=0\}} {\ind_\xi \bar\nabla h}
 \\ =&\sum_{\{\xi \in\partial M:\bar\nabla h(\xi )=0,\, \partial_\nu u(\xi,0 )>0\}} {\ind_\xi \bar\nabla h}+\sum_{\{\xi \in\partial M:\bar\nabla h(\xi )=0,\, \partial_\nu u(\xi )<0\}} {\ind_\xi \bar\nabla h}.
\end{aligned} \end{equation}
Finally by summing \eqref{s1}, \eqref{s2} and \eqref{s3}, we get
\begin{equation}\label{s4}\begin{aligned} 
 \sum_{\{x \in M:\nabla u(x)=0\}} {\ind_x \nabla u}=\frac12\left( \chi\(\mathscr DM\)+\chi(\partial M)\right)-\sum_{\{\xi \in\partial M:\bar\nabla h(\xi )=0,\, \partial_\nu u(\xi)<0\}} {\ind_\xi \bar\nabla h}.
\end{aligned} 
\end{equation}
Using
$$ 
\chi(\mathscr DM)=2\chi(M)-\chi(\partial M)
$$
in \eqref{s4}, we establish \eqref{E0}. The proof is now concluded.
\end{proof}

We point out that if both $u$ and $h$ are {\it Morse} functions, formula \eqref{E0} can be made more precise. In fact, in this case, the index of a non-degenerate critical point $x$ is $(-1)^{\mathtt m(x)}$ where $\mathtt m(x)$ is the Morse index of $x$, and formula \eqref{E0} reads as
\begin{equation}\label{E1}\begin{aligned} 
 \sum_{\{x \in M:\nabla u(x )=0\}} {(-1)^{\mathtt m(x)}}= \chi(M)-\sum_{\{\xi \in\partial M:\bar\nabla h(\xi )=0,\, \partial_\nu u(\xi )<0\}} {(-1)^{\mathtt m(\xi)}}
\end{aligned} \end{equation}
 
In particular, when $n=2$ (i.e., for Riemannian surfaces), taking into account that all the critical points in the interior are saddle points and on the boundary are maxima or minima, we deduce the following corollary
\begin{corollary}
Under the assumptions of Theorem \ref{T_0} we have that
$$\begin{aligned}\sharp \{\hbox{critical points of $u$}\}=&\ \ \ \sharp \{\hbox{minima of $u$ on $\partial M$ with $ \partial_\nu u(\xi )<0$}\}\\ &-\sharp \{\hbox{maxima of $u$ on $\partial M$ with $ \partial_\nu u(\xi )<0$}\}\\ &-\chi(M).\end{aligned}
$$ 
\end{corollary}

\section{A generic formula for the critical points of Steklov eigenfunctions}\label{S1}

In this section we apply Theorem \ref{T_0} to the case of Steklov eigenfunctions and prove that for generic metrics on a given manifold, formula \eqref{main1} holds, along with its consequences (formulas \eqref{main2} and \eqref{main3}). 

\medskip

\begin{theorem}\label{T_2}
Let $M$ be a smooth compact $n$-dimensional manifold with smooth boundary $\partial M$ and let $G$ be the set of smooth Riemannian metrics on $M$. Then for a residual (hence dense) subset of $G$ of smooth metrics we have that non-constant Steklov eigenfunctions $u$ satisfy
\begin{equation}\label{E2_1}
\sum_{\{x\in M:\nabla u(x)=0\}}{\ind_x}\nabla u=\chi(M)-\sum_{\{\xi\in\partial M:\bar\nabla u(\xi)=0, u(\xi)<0\}}{\ind_{\xi}}\bar\nabla u.
\end{equation}
Moreover, if
$$
\mathscr P:=\{\xi\in\partial M: u(\xi)<0\}
$$
we have
\begin{equation}\label{E2_2}
\sum_{\{x\in M:\nabla u(x)=0\}}{\ind_x}\nabla u=\chi(M)-\chi(\mathscr P).
\end{equation}
\end{theorem}
Before proving Theorem \ref{T_2}, we discuss a few consequences.

\medskip

Assume that $\mathscr P$ is the union of $\ell$ disjoint connected components $\mathscr P_i$. Then by \eqref{E2_2} we get
\begin{equation}\label{E2_3}
\sum_{\{x\in M:\nabla u(x)=0\}}{\ind_x}\nabla u=\chi(M)-\sum_{i=1}^{\ell}\chi(\mathscr P_i).
\end{equation}
Since $u$ and its restriction on $\partial M$ are generically Morse (see \cite{wang} and Theorem \ref{T_gen}), we have from \eqref{E2_1} that
$$ \sum_{\{x \in M:\nabla u(x )=0\}} {(-1)^{\mathtt m(x)}}= \chi(M)-\sum_{\{\xi \in\partial M:\bar\nabla u(\xi )=0,\, u(\xi )<0\}} {(-1)^{\mathtt m(\xi)}}
.
$$

\medskip

If $n=2$ (i.e., for Riemannian surfaces), each $\mathscr P_i$ is either an interval (hence $\chi(\mathscr P_i)=1$) or a simple closed curve (hence $\chi(\mathscr P_i)=0$). Therefore we deduce that, generically
\begin{equation}\label{E2_4}
\sum_{\{x\in M:\nabla u(x)=0\}}{\ind_x}\nabla u=\chi(M)-\sum_{j=1}^{L}\ell_j,
\end{equation}
where $L$ is the number of connected components $\Gamma_j$ of $\partial M$ where $u$ changes sign and $\ell_j$ is the number of times that $u$ changes sign on $\Gamma_j$. Moreover, we also have
$$\begin{aligned}\sharp \{\hbox{critical points of $u$}\}=&\ \ \ \sharp \{\hbox{minima of $u$ on $\partial M$ with $ u(\xi )<0$}\}\\ &-\sharp \{\hbox{maxima of $u$ on $\partial M$ with $ u(\xi )<0$}\}\\ &-\chi(M)\\
 =& \sum_{j=1}^L\ell_j-\chi(M).\end{aligned}$$

\begin{proof}[\bf Proof of Theorem \ref{T_2}]
We apply Theorem \ref{T_0} to a non-constant Steklov eigenfunction $u$. To do so, we verify that hypotheses $i)$, $ii)$ and $iii)$ of Theorem \ref{T_0} are satisfied by $u$ and $h=u_{|_{\partial M}}$ for a generic metric $g$ on $M$. In \cite{wang} it is proved that, generically, restrictions of Steklov eigenfunctions on $\partial M$ are Morse functions, hence $ii)$ is verified. In \cite{wang} it is also proved that generically there are no singular zeros for $u_{|_{\partial\Omega}}$, which means that if $u(\xi)=0$ for some $\xi\in\partial M$, then $\bar\nabla u(\xi)\ne 0$. Since $u(\xi)=0$ if and only if $\partial_{\nu}(\xi)=0$ by the Steklov condition, we deduce that if $\partial_{\nu}u(\xi)=0$, then $\bar\nabla u(\xi)\ne 0$. Hence also $iii)$ is generically satisfied. On the other hand, in \cite{wang} generic properties of Steklov eigenfunctions in $M$ are not discussed. We prove in Appendix \ref{gen} (see Theorem \ref{T_gen}) that $u$ is generically a Morse function in $M$, which implies $i)$. Therefore Theorem \ref{T_0} applies, and since $\partial_{\nu}u(\xi)<0$ for $\xi\in\partial\Omega$ if and only if $u(\xi)<0$ by the Steklov condition, we deduce \eqref{E2_1} from \eqref{E0}.

\medskip

In order to prove \eqref{E2_2}, we use Poincaré-Hopf Theorem on $\mathscr P$, which is the union of $n-1$-dimensional connected manifolds $\mathscr P_i$, possibly with boundary. If $\mathscr P_i$ has boundary (which is a zero level set of $u_{|_{\partial M}}$), from $iii)$ we deduce that generically $\bar\nabla u\ne 0$ on $\partial\mathscr P_i$, which implies that $\langle\bar\nabla u,N\rangle>0$, where $N$ is the outer unit normal to $\mathscr P_i$. Hence we can apply Theorem \ref{PH} and get that
\begin{equation}\label{PH_dim}
\sum_{\{\xi\in\partial M:\bar\nabla u(\xi)=0,u(\xi)<0\}}{\ind_{\xi}}\bar\nabla u=\sum_{i=1}^{\ell}\chi(\mathscr P_i)=\chi(\mathscr P).
\end{equation} 
Now \eqref{E2_2} follows from \eqref{E2_1} and \eqref{PH_dim}.
\end{proof}

\section{The case of the second eigenfunction}\label{S2}

In this section we consider the second Steklov eigenfunction $u_2$ on {\it any} smooth connected Riemannian surface $M$ of genus $0$, and we prove that \eqref{E2_4} holds (i.e., the result is not generic). To do so we prove a few properties of $u_2$ when $n=2$. Note the result applies, in particular, to Euclidean domains. Throughout all this section, $M$ is a genus $0$ surface with boundary.

\begin{theorem}\label{u2_morse}
The function $u_2$ is a Morse function.
\end{theorem}

\begin{proof}
The proof follows the same lines as that of \cite[Lemma 3]{PP}. Let $u_2$ be a second eigenfunction and let $p\in M$ be a critical point of $u_2$. Assume it is degenerate. This implies that the Hessian at $p$ vanishes. Consider the function $w=u_2-u_2(p)$. Clearly $\Delta w=0$, and $p$ is a zero of $w$ such that $\nabla w(p)=0$, $D^2w(p)=0$. Hence the function $w$ has at least {\it three} nodal curves intersecting at $p$ and forming equal angles. Consider a small geodesic disk $D$ centered at $p$ such that the nodal curves divide $D$ into $2k\geq 6$ disjoint regions $D_i$, $i=1,...,2k$. Assume by contradiction that there are at most $3$ nodal domains for $u_2$ in $M$. Then, necessarily, at least three of the disjoint regions $D_i$ must belong to the same nodal domain $\Omega$. We label them by $D_1,D_2,D_3$. They are not adjacent, since $u_2$ is harmonic and cannot have interior local maxima or minima. Consider now  a simple closed curve $\gamma_{12}$ passing through $p$, belonging to the nodal domain $\Omega$, and intersecting both $D_1$ and $D_2$. Since $D_1$ an $D_2$ are not adjacent, one of the regions $D_i$, $i\geq 4$, belongs to a nodal domain $\Omega'$ distinct from $\omega$. The nodal domain $\Omega'$ must be contained in the bounded region enclosed by $\gamma_{12}$. This is true since we are in genus $0$ (we may always think of $M$ as a sphere minus $m$ disjoint geodesic disks, where $m$ is the number of boundary components). We repeat the same argument by forming closed loops $\gamma_{13}$ and $\gamma_{23}$ considering the couples $D_1,D_3$ and $D_2,D_3$. We can choose these curves such that $\gamma_{12}\cap\gamma_{13}=\gamma_{12}\cap\gamma_{23}=\gamma_{13}\cap\gamma_{23}=p$. We conclude that there are at least other two  nodal domains $\Omega'',\Omega'''$, enclosed by the curves $\gamma_{13}$ and $\gamma_{23}$ respectively, and disjoint from $\Omega$ and $\Omega'$. This is a contradiction with the fact that we assumed that the nodal domains were at most $3$.


We have proved that there are at least $4$ interior nodal domains for $w$.

Let then $\Omega_i$, $i=1,...,m$, $m\geq 4$, the interior nodal domains of $w$, and let $w_i=w|_{\Omega_i}$ (extended by $0$ outside $\Omega_i$). Let
$$
\phi:=\sum_{i=1}^m a_iw_i
$$
where $a_i\in\mathbb R$ are not all zero and are chosen such that
$$
\int_{\partial M}\phi=\sum_{i=1}^ma_i\int_{\partial\Omega_i}w_i=0
$$
and
$$
\sum_{i=1}^ma_i^2\int_{\partial\Omega_i}w_i=0.
$$
Here $\partial\Omega_i=\partial M\cap\overline\Omega_i$. By construction, $\phi\in H^1(M)$. It is always possible to find such $a_i$ since $m\geq 4$. In fact, let us set $\gamma_i:=\int_{\partial\Omega_i}w_i$. We need to solve the system
$$
\begin{cases}
\sum_{i=1}^m\gamma_ia_i=0\\
\sum_{i=1}^m\gamma_1 a_i^2=0.
\end{cases}
$$
Without loss of generality, we may assume that $\gamma_i>0$ for $i$ odd and $\gamma_i<0$ for $i$ even. We will show that a non-trivial solution exists. We set $a_i=0$ for $i\geq 5$ and look for a non-trivial solution to
$$
\begin{cases}
|\gamma_1|a_1-|\gamma_2|a_2+|\gamma_3|a_3-|\gamma_4|a_4=0\\
|\gamma_1|a_1^2-|\gamma_2|a_2^2+|\gamma_3|a_3^2-|\gamma_4|a_4^2=0
\end{cases}
$$
which we re-write as
$$
\begin{cases}
|\gamma_1|a_1-|\gamma_2|a_2+|\gamma_3|a_3-|\gamma_4|a_4=0\\
(\sqrt{|\gamma_1|}a_1+\sqrt{|\gamma_2|}a_2)(\sqrt{|\gamma_1|}a_1-\sqrt{|\gamma_2|}a_2)\\\quad\quad\quad\quad\, +(\sqrt{|\gamma_3|}a_3+\sqrt{|\gamma_4|}a_4)(\sqrt{|\gamma_3|}a_3-\sqrt{|\gamma_4|}a_4)=0
\end{cases}
$$
A non-trivial solution is given, for example, by any solution of the following system of three equations in four unknowns $a_1,a_2,a_3,a_4$:
$$
\begin{cases}
|\gamma_1|a_1-|\gamma_2|a_2+|\gamma_3|a_3-|\gamma_4|a_4=0\\
\sqrt{|\gamma_1|}a_1+\sqrt{|\gamma_2|}a_2=0\\
\sqrt{|\gamma_3|}a_3+\sqrt{|\gamma_4|}a_4=0
\end{cases}.
$$
Now, since $\int_{\partial M}\phi=0$, by the min-max principle we have
$$
\sigma_2\leq\frac{\int_M|\nabla \phi|^2}{\int_{\partial M}\phi^2}=\frac{\sum_{i=1}^m\alpha_i^2\int_{\Omega_i}|\nabla w|^2}{\sum_{i=1}^m\int_{\partial\Omega_i}w^2}.
$$
We also have
$$
\int_{\Omega_i}|\nabla w|^2=\int_{\partial\Omega_i}\partial_{\nu}ww=\sigma_2\int_{\partial\Omega_i}w^2+\sigma_2 u_2(p)\int_{\partial\Omega_i}w,
$$
and hence
$$
\sum_{i=1}^m\alpha_i^2\int_{\Omega_i}|\nabla w|^2=\sigma_2\sum_{i=1}^m\int_{\partial\Omega_i}w^2+\sigma_2 u(p)\sum_{i=1}^m\alpha_i^2\int_{\partial\Omega_i}w=\sigma_2\sum_{i=1}^m\int_{\partial\Omega_i}w^2
$$
by the specific choice of $\alpha_i$. Therefore
$\phi$ realizes the equality in the min-max principle, which implies that it is a second Steklov eigenfunction. On the other hand, $\phi$ has at least four nodal domains, which is not possible since a second Steklov eigenfunction has exactly two interior nodal domains.
\end{proof}

\begin{theorem}\label{bdry_zeros}
Singular zeros of $u_2$ on the boundary are non-degenerate.
\end{theorem}
\begin{proof}
Let $x_0\in\partial M$ be such that $u_2(x_0)=0$ and $\nabla u_2(x_0)=0$. Hence, in a neighborhood of $x_0$ the zero level set of $u_2$ consists of at least two curves intersecting at equal angles $\theta\in(0,\pi)$ in $x_0$, and moreover the smallest angle between these curves and the tangent to $\partial M$ at $x_0$ is $\theta/2$. This result follows from the local behavior of Steklov eigenfunctions on surfaces near a singular boundary zero, see \cite[Appendix A]{GH} for more details and more general boundary conditions. In particular, if $D^2u_2(x_0)=0$, we have at least three zero level curves intersecting in $x_0$, and by the same argument used in the proof of Theorem \ref{u2_morse}, this implies that there are at least three nodal domains for $u_2$, which is impossible. Hence $D^2u_2(x_0)\ne 0$, and $x_0$ is a non-degenerate critical point.
\end{proof}

\begin{remark}
In the case of a simply-connected surface, the proof of the previous theorem tells that there are no singular zeros of $u_2$ on $\partial M$. The natural question arises whether in the non simply-connected case singular zeros can occur for $u_2$ at the boundary. The answer is positive, as we can see in the case of a specific flat cylinder (see the examples at the end of this section). In this sense Theorem \ref{bdry_zeros} is sharp.
\end{remark}

\begin{remark}
In view of Theorem \ref{bdry_zeros}, we see that for a Steklov eigenfunction on a genus $0$ surface, given a point $x_0\in\partial M$ where $u_2$ vanishes, two possibilities may occur: $x_0$ is not singular, hence in a neighborhood of $x_0$ there is a unique zero level curve meeting $\partial M$ orthogonally at $x_0$; $x_0$ is singular, and not degenerate, hence in a neighborhood of $x_0$ the zero level set of $u_2$ is given by the union of two curves, meeting orthogonally at $x_0$, and forming with the tangent to $\partial M$ at $x_0$ an angle of $\pi/4$.
\end{remark}

For completeness, we state and prove the following result on interior singular zeros. We will not need it for the purposes of the present section. However, in Appendix \ref{gen} we prove that the analogous result holds for generic metrics and for all eigenfunctions. Note that it is not a generic result for $u_2$.

\begin{theorem}
The function $u_2$ has no interior singular zeros.
\end{theorem}

\begin{proof}
Suppose $x_0\in M$ is such that $u_2(x_0)=0$, $\nabla u_2(x_0)=0$. We deduce that at $x_0$ the zero level set of $u_2$ consists of at least two curves intersecting at equal angles. This is a classical result, and we may refer e.g., to \cite[Appendix A]{GH} for a more detailed discussion. As for Theorems \ref{u2_morse} and \ref{bdry_zeros}, this implies that there are at least {\it three} interior nodal domains for $u_2$, which is not possible since $u_2$ has exactly {\it two} interior nodal domains.
\end{proof}

We now state the following result which can be deduced as a particular case of \cite[Theorem 3]{MM}.

\begin{theorem}\label{morse}
Let $M$ be a smooth surface with boundary $\partial M$, and let $V$ be a smooth vector field on $M$ with isolated singularities $x_i$ and without singularities on $\partial M$. Assume also that the projection of $V$ on each connected component $\mathscr Q_i$ of $\mathscr Q=\{\xi\in\partial M:\langle V,\nu\rangle>0\}$ is outward pointing to $\partial\mathscr Q_i$ (if $\mathscr Q_i$ has a boundary). Then
$$
\sum_i {\ind_{x_i} V}=\chi(M)-\chi(\mathscr Q).
$$
Here $\nu$ is the outer unit normal to $\partial M$.
\end{theorem}
If $V=\nabla u_2$ then $\mathscr Q=\{\xi\in\partial M: u(\xi)>0\}$ and for $n=2$, $\chi(\mathscr Q)=\chi(\mathscr P)$, where  $\mathscr P=\{\xi\in\partial M: u(\xi)<0\}$. If $u_2$ satisfies the hypothesis of Theorem \ref{morse}, we have \eqref{E2_4}.

We note that Theorem \ref{morse} applies to $u_2$ for genus $0$ surfaces in case we have no singular zeros on $\partial M$, since interior critical points are Morse by \ref{u2_morse}. However, it applies also in the case when there are boundary singular points, since by Theorem \ref{bdry_zeros} they are non degenerate, and the behavior of $u_2$ in a neighborhood of such points is explicit (see also \cite[Appendix A]{GH}). To see this, let us consider for simplicity the Euclidean case to which one can reduce with a conformal diffeomorphism. Let $x_0$ be a singular boundary zero. We can assume $x_0=(0,0)$ (in Euclidean coordinates $(x,y)$), and also that the tangent to the boundary at $x_0$ is given by $y=0$. Hence, in a neighborhood of $x_0$ we have that (up to a sign) $u_2(x,y)=c(x^2-y^2)+O((x^2+y^2)^{1+\epsilon})$, $c>0$. We can replace $\nabla u_2$ by $V$, where $V=\nabla u_2+(\epsilon-\sqrt{x^2+y^2})\chi_{B(x_0,\epsilon)}\vec e_2$. Here $(\vec e_1,\vec e_2)$ denotes the canonical basis of $\mathbb R^2$. If $\epsilon$ is chosen sufficiently small, $V$ has no zeros on $B(x_0,\epsilon)\cap\partial M$, $\langle V,\nu\rangle>0$ in $B(x_0,\epsilon)\cap\partial M$, and no new interior critical points have been produced in $B(x_0,\epsilon)\cap M$. Hence one can apply Theorem \ref{morse} to $V$. Doing so we are counting the sum of the indexes of the {\it interior} critical points of $u_2$. Note that for the perturbed vector field the point $x_0$ is an interior point of a connected component of $\mathscr Q$. We have just proved the following non-generic result.

\begin{corollary}\label{cor_sur}
Let $M$ be a smooth compact surface of genus $0$ with boundary and let $u_2$ be the second Steklov eigenfunction on $M$. Then
\begin{equation}\label{E_c}
\sharp\{{\rm interior\ critical\ points\ of\ }u_2\}=\sum_{j=1}^L\ell_j-\chi(M),
\end{equation}
where $L$ is the number of connected components of $\partial M$ where $u_2$ changes sign exactly $\ell_j$ times (in the case of singular boundary zeros, they must not be counted as sign changing).

\end{corollary}

\begin{remark}
We note that Theorem \ref{T_2} in the case $n=2$ can be also deduced from Theorem \ref{morse}. In  fact, using $-u$ in \eqref{E2_1} we get the formula 
$$
(-1)^n\sum_{\{x\in M:\nabla u(x)=0\}}{\rm ind}_x\nabla u=\chi(M)-\chi(\mathscr Q),
$$
where $\mathcal Q=\{\xi\in\partial M : u(\xi)>0\}$.

 However the proof of Theorem \ref{morse} (and of the more general \cite[Theorem 3]{MM}) is much longer and technical with respect to the proof of Theorem \ref{T_0} using the double manifold. The general version of Theorem \ref{morse}, \cite[Theorem 3]{MM}, is a generalization of the Poincaré-Hopf theorem for manifolds with boundary when the vector field is not outward pointing everywhere on the boundary.
\end{remark}

We can verify the validity of Corollary \ref{cor_sur} in various situations. For example, let $A=\{x\in\mathbb R^2:r<|x|<R\}$ for some $0<r<R$ be a planar annulus. For any $r,R$ we have that $u_2$ changes sign exactly once on $|x|=r$ and on $|x|=R$, and the boundary zeros are not singular. Hence
$$
\sharp\{{\rm critical\ points\ of\ }u_2\}=2-0=2:
$$
for any annulus of the plane, we have exactly two interior critical points, which are saddles. This can be checked by explicitly computing $u_2$.

\medskip

Let us consider now $C_T=\mathbb S^1\times[-T,T]$, a flat cylinder. It is well-known that the shape of $u_2(\theta,z)$ depends on $T$ (see e.g., \cite{GP_survey} where the explicit expressions of Steklov eigenvalues and eigenfunctions on flat cylinders are given). Let $T^*$ be the unique positive root of $T\tanh(T)=1$. Then, it is easy to see that for $0<T<T^*$, the second eigenvalue is double and is given by $\tanh(T)$, and any second eigenfunction $u_2$ changes sign once on each boundary component $\mathbb S^1\times\{-T\}$ and $\mathbb S^1\times\{T\}$. As for the case of the annulus, we have two interior critical points, which are saddles. For $T>T^*$ however the second eigenvalue is simple and given by $\frac{1}{T}$, and $u_2(\theta,z)=z$. It has constant sign on each connected component of the boundary and, therefore, from \eqref{E_c} we deduce that it has no interior singular points (which is an immediate check). We observe then that only the connected components of the boundary where the eigenfunction changes sign influence the number of interior critical points.

Consider now the flat cylinder $C_{T^*}$.  The second eigenvalue has multiplicity three and three linearly independent eigenfunctions are given by $\cosh(z)\cos(\theta)$, $\cosh(z)\sin(\theta)$, $z$. Hence $\cosh(z)\cos(\theta)\pm\frac{\cosh(T^*)}{T^*}z$ is also an eigenfunction. It has only two critical points: $(\theta,z)=(\pi/2,\pm T^*)$ and $(\theta,z)=(\frac{3\pi}{2},\mp T^*)$. These are boundary critical points which are non-degenerate, in particular, they are singular boundary zeros. In correspondence of these singular boundary zeros, nodal lines meet with an angle of $\pi/2$. Hence there are no interior critical points, and Corollary \ref{E_c} is verified. In general, for  $C_{T^*}$ one can check the validity of Corollary \ref{E_c} for any eigenfunction of the form $\cosh(z)\cos(\theta)+cz$, $c\in\mathbb R$. In fact, for $|c|<\frac{\cosh(T^*)}{T^*}$ we have exactly two interior critical points, and the eigenfunction changes sign exactly once on each boundary circle; for $|c|>\frac{\cosh(T^*)}{T^*}$ there are no critical points (in the interior and on the boundary), and the function has constant sign on each boundary circle.

\appendix

\section{Generic results for Steklov eigenfunctions}\label{gen}

Let $M$ be a compact smooth $n$-dimensional manifold with smooth boundary $\partial M$, and let $G$ denote the set of smooth Riemannian metrics on $M$. In this section we will prove the following theorem.

\begin{theorem}\label{T_gen}
Let $M$ be a compact smooth $n$-dimensional manifold with smooth boundary $\partial M$. Then the subset of $G$ of smooth metrics on $M$ for which all Steklov eigenfunctions are Morse functions in $M$ is residual (hence dense) in $G$.
\end{theorem}

We prove Theorem \ref{T_gen} following the ideas of \cite{U} (see also \cite{wang}).

We shall need the following transversality theorem (see \cite[Transversality Theorem 2]{U}): 

\begin{theorem}\label{trans2}
Let $Q,B,X,Y,Y'$ be separable Banach manifolds, $Y'\subset Y$, $X,Y$ finite dimensional. Let $\pi:Q\to B$ be a $C^k$ Fredholm map of index $0$. Then, if $F:Q\times X\to Y$ is a $C^k$ map for $k>\max\{1,{\rm dim}X+{\rm dim}Y'-{\rm dim}Y\}$ and $F$ is transverse to $Y'$, the set $\{b\in B:F_b:=F_{\pi^{-1}b}{\rm\ is\ transverse\ to\ }Y'\}$ is residual in $B$.
\end{theorem}

Let $H^k(\partial M)$ denote the usual Sobolev spaces of functions in $L^2(\partial M)$ with weak derivatives up to order $k\in\mathbb N$ in $L^2(\partial M)$. We choose $k$ sufficiently large so that functions in $H^k(\partial M)$ have some regularity (for example, $C^1$). Following \cite{U}, we consider the map
$$
\varphi: H^k(\partial M)\times \mathbb R\times G\to H^{k-1}(\partial M)
$$
defined by
$$
\varphi(f,\sigma,g)=(\mathscr D-\sigma)f,
$$
where $\mathscr D$ is the Dirichlet-to-Neumann map, namely $\mathscr D f=\partial_{\nu}(Hf)$, where $Hf$ is the {\it harmonic extension} of $f$, namely, it solves
$$
\begin{cases}
\Delta (Hf)=0 & {\rm in\ }M,\\
Hf=f & {\rm on\ }\partial M.
\end{cases}
$$
The Laplacian is taken with respect to the metric $g$.

\medskip

Since we are considering smooth metrics, from \cite[Lemma 2.1]{U} we get that $\varphi_b:=\varphi(\cdot,\cdot,b)$ is a Fredholm map of index $0$. Moreover, it is clear that $\varphi(f,\sigma,g)=0$ if and only if $f$ is an eigenfunction of $\mathscr D$ with eigenvalue $\sigma$ for the metric $g$ (see also \cite[Lemma 2.2]{U}).

\medskip

Consider now the Banach manifold $Q=\varphi^{-1}(0)$, namely the subset of $H^k(\partial M)\times\mathbb R\times G$ of all the eigenfunctions, and consider the following map:
$$
\beta:Q\times M\to TM
$$
 defined by
 $$
 \beta(f,\sigma,g,x)=\nabla Hf(x).
 $$
 Note that $f$ is an eigenfunction of $\mathscr D$ with eigenvalue $\sigma$ for the metric $g$ if and only if $Hf$ is a Steklov eigenfunction on $M$ for the same metric and with the same eigenvalue. From the smoothness assumptions on the metrics of $G$, we have that $\beta$ is a smooth map.
 
 \medskip
 
We will denote the differentiation of $\varphi$ with respect to the first, second and third parameters respectively by $D_1,D_2,D_3$. It is convenient now to describe the tangent space of $Q$:
\begin{equation}\label{tangent}
 T_{(f,\sigma,g)}Q=\{(v,s,h)\in H^k(\partial M)\times \mathbb R\times T_gG:(\mathscr D-\sigma)v+sf+(D_3\varphi|_{(f,\sigma,g)})(h)=0\}.
\end{equation}

The following two propositions proved in \cite[Lemma 3.7 and Proposition 3.9]{wang} (see also \cite{U}) will be crucial in the sequel:

\begin{proposition}\label{density0}
Let $(f,\sigma,g)\in Q$ with $\sigma\ne 0$, and consider the map $D_3\varphi|_{(f,\sigma,g)}:T_gG\to H^{k-1}(\partial M)$. Then
$$
({\rm Im}\,D_3\varphi|_{(f,\sigma,g)})^{\perp}\subseteq\{\psi\in H^{k-1}(\partial M):{\rm supp}(\psi)\subseteq f^{-1}(0)\},
$$ 
or, equivalently,
$$
\{\psi\in H^{k-1}(\partial M):{\rm supp}(\psi)\cap f^{-1}(0)=\emptyset\}\subseteq {\rm Im}\,D_3\varphi|_{(f,\sigma,g)}.
$$
In particular, for any $\psi\in H^{k-1}(\partial M)$ with ${\rm supp}(\psi)\cap f^{-1}(0)=\emptyset$, there exists a smooth function $\omega$ defined on $M$ such that
$$
D_3\varphi|_{(f,\sigma,g)}(\omega g)=\psi.
$$
\end{proposition}
Using the previous proposition we get the following:
\begin{proposition}\label{density1}
${\rm Im}\,D\varphi|_{(f,\sigma,g)}=H^{k-1}(\partial M)$ for any $(f,\sigma,g)\in Q$ with $\sigma\ne 0$. In particular $0$ is a regular value of $\varphi$.
\end{proposition}

We are now in position to prove Theorem \ref{T_gen}.

\begin{proof}[Proof of Theorem \ref{T_gen}]
We wish to apply Theorem \ref{trans2} with $Q=\varphi^{-1}(0),$ $B=G$, $X=M$, $Y=TM$, $F=\beta$. Note that $\beta$ is smooth, and the projection map $\pi:Q\to G$ (restriction to $Q$ of the projection $\pi:H^k(\partial M)\times \mathbb R\times G\to G$) has Fredholm index $0$ (see \cite[Lemma 2.5]{U}). Hence we have to prove that $\beta$ is transverse to $Y'=0$ the zero section of $TM$ (i.e., the zero vector field on $M$).

\medskip

To do so, let $\beta_x:=\beta(\cdot,\cdot,\cdot,x):Q\to T_xM$. We need to prove that $D\beta_x:T_{(f,\sigma,g)}\to T_xM$ is surjective when $\beta_x(f,\sigma,g)=0$ (and $\sigma\ne 0$). Writing $D\beta_x$ explicitly, since $H$ is linear, we have
$$
D\beta_x|_{(f,\sigma,g)}(v,s,h)=\nabla Hv(x).
$$
We recall that $f$ is an eigenfunction of $\mathscr D$ with eigenvalue $\sigma$ for the metric $g$. We also assume that $\sigma\ne 0$ (that is, we consider non-constant eigenfunctions). Assume now that there exists $0\ne V\in T_xM$ such that $\langle \nabla Hv(x),V \rangle =0$ for all $v$ such that $(v,s,h)\in T_{(f,\sigma,g)}Q$. Let us consider the following two situations:
\begin{itemize}
\item $(f,0,0)\in T_{(f,\sigma,g)}Q$ (in fact, it satisfies \eqref{tangent}). Hence $D\beta_x|_{(f,\sigma,g)}(f,0,0)=\nabla Hf(x)$ and by assumption $\langle \nabla Hf(x),V\rangle=0$.
\item From Proposition \ref{density0} it follows that for all $\psi\in H^{k-1}(\partial M)$ with ${\rm supp}(\psi)\cap f^{-1}(0)=\emptyset$, there exists a conformal variation of the metric $\omega g$ such that $D_3\varphi|_{f,\sigma,g}(\omega g)=\psi$ ($\omega$ is a smooth function on $M$). Let $\psi\in{\rm Ker}\,(\mathscr D-\sigma)^{\perp}$ be such that ${\rm supp}(\psi)\cap f^{-1}(0)=\emptyset$, and let $\omega$ the corresponding conformal factor from Porposition \ref{density0}. We also have that $\psi=\sum_{j:\sigma_j\ne\sigma}a_jf_j$, where $\{f_j\}$ is a $L^2(\partial M)$ orthonormal basis of $H^{1/2}(\partial M)$ of eigenfunctions of $\mathscr D$. Let us define
$$
v:=-\sum_{\sigma_j\ne\sigma}\frac{a_jf_j}{\sigma_j-\sigma}.
$$
Then $(v,0,\omega g)\in T_{f,\sigma,g}Q$, in fact $(\mathscr D-\sigma)v+D_3\phi|_{f,\sigma,g}(\omega g)=-\psi+\psi=0$, so that the condition in \eqref{tangent} is satisfied. Hence
$$
\langle\nabla v(x),V\rangle=\sum_{\sigma_j\ne\sigma}\frac{a_j}{\sigma_j-\sigma}\langle\nabla Hf_j(x),V\rangle=0.
$$
We recall that $a_j=\int_{\partial M}\psi f_j$. Since $\psi\in{\rm Ker}\,(\mathscr D-\sigma)^{\perp}$ with ${\rm supp}(\psi)\cap f^{-1}(0)=\emptyset$ is dense in ${\rm Ker}\,(\mathscr D-\sigma)^{\perp}$ with respect to the $L^2(\partial M)$ norm, we get that $\langle\nabla Hf_j(x),V\rangle=0$ for all $j$ with $\sigma_j\ne\sigma$.
\end{itemize}
In conclusion, assuming that $D\beta_x$ is not surjective, we get that there exists $0\ne V\in T_xM$ such that
$$
\langle \nabla H(x),\nabla V\rangle=0
$$
for all harmonic functions $H$ in $M$. The proof is concluded if we prove that for any $x\in M$ and any $0\ne V\in T_xM$ there exists a harmonic function $H$ such that $\nabla H(x)=V$. We prove this fact in Proposition \ref{prop_grad} here below.
\end{proof}

\begin{proposition}\label{prop_grad} For any $x\in M$, $0\ne V\in T_xM$ there exists a harmonic function $H$ in $M$ with $\nabla H(x)=V$. 
\end{proposition}
 
One can check that the proposition is true when $M$ is a simply connected surface. In fact, $M$ can be conformally mapped to $B(0,1)\subset\mathbb R^2$. In particular, for any fixed $x\in M$, there is always a unique (up to rotations) $\phi:M\to B(0,1)$ conformal diffeomorphism such that $\phi(x)=0$. Then, taking $H_i=x_i\circ\phi$ ($x_i$ coordinate functions in $\mathbb R^2$), we have that $\Delta H_i=0$ and $\{\nabla H_i(x)\}_{i=1}^2$ span $T_xM$. In general, the proposition is true for any domain of $\mathbb R^n$ by just taking the coordinate functions (in a coordinate system where $x$ is the origin). Also, it is immediately true for any surface of genus $g$ and $k$ boundary components by the Uniformization theorem. For a higher dimensional Riemannian manifold we need a slightly more involved proof.

\begin{proof}[Proof of Proposition \ref{prop_grad}]
It is a standard fact to prove that for any $x\in M$ and any $0\ne V\in T_xM$, there exists a neighborhood of $x$ (which can be taken open, smooth) and a smooth $\varphi:U\to\mathbb R$, continuous on $\overline U$, such that $\Delta\varphi=0$ in $U$. For example, just take (a linear combination of) harmonic coordinates in a neighborhood of $x$. 

\medskip

Let now $U\subset M$ be a smooth open subset of $M$. For any $m\geq 2$ let
$$
X=\{v\in H^m(U):\Delta v=0\}
$$
and
$$
Y=\{u\in H^m(M):\Delta u=0\}.
$$
We prove that, given $v\in X$ and $\epsilon>0$ there exists $u\in Y$ such that
$$
\|v-u_{|_{U}}\|_{H^m(U)}<\epsilon.
$$
If we take $m$ sufficiently large, then by Sobolev embedding we have just that any harmonic function in $v$ in $H^m(U)\subset C^1(U)$ can be approximated in the $C^1(U)$-norm by harmonic functions in $H^m(M)\subset C^1(M)$, and this is enough to prove the proposition.

\medskip

Consider the map $R:H^{m-1/2}(\partial M)\to X$ given by
$$
Rg=u_{|_{U}}\,,
$$
where $u\in H^m(M)$ solves $\begin{cases}\Delta u=0 & {\rm in\ M}\\u=g & {\rm on\ }\partial M\end{cases}$. We need to prove that ${\rm Im}(R)\subset X$ is dense in $X$ in the $H^m(M)$ norm. By the Hahn-Banach Theorem, it is enough to prove that for any linear functional $T:H^m(U)\to\mathbb R$ such that $T(Rg)=0$ for all $g\in H^{m-1/2}(\partial M)$, we have $T(v)=0$ for all $v$ in $X$.

\medskip

By duality, for any such $T$ there exists $h\in H^{-m}(U)$ such that $T(\cdot)=\int_{U}h(\cdot)$.

\medskip Let $\tilde h\in H^{-m}(M)$ be defined by
$$
\tilde h=\begin{cases}
h & {\rm in\ } U\\
0 & {\rm in\ } M\setminus\overline U.
\end{cases}
$$ 
and let $w\in H^{-m+2}(M)$ be the solution to
$$
\begin{cases}
-\Delta w=\tilde h & {\rm in\ }M\\
w=0 & {\rm on\ }\partial M.
\end{cases}
$$
(in the sense that $\int_M w(-\Delta\phi)=\int_M\tilde h\phi$ for all $\phi\in H^m(M)\cap H^1_0(M)$).

\medskip

Let $g\in H^{m-1/2}(\partial M)$ and $u\in H^m(M)$ be such that $\Delta u=0$ in $M$, $u=g$ on $\partial M$. 

\medskip

Assume that $T(Rg)=0$ for all $g\in H^{m-1/2}(\partial M)$ (recall that $Rg=u_{|_U}$). Since $w$ can be approximated in the $H^{-m+2}(M)$ norm by functions in $C^{\infty}_c(M)$, we deduce that
$$
T(Rg)=\int_U u h=\int_M\tilde h u=\int_{\partial M}\partial_{\nu}w g.
$$
In the last integral the duality is between $H^{-m+\frac{1}{2}}$ and $H^{m-\frac{1}{2}}$. Hence, since $T(Rg)=0$ for all $g$ we deduce that $\partial_{\nu}w=0$ in $H^{-m+\frac{1}{2}}(\partial M)$. However, $\tilde h\equiv 0$ in a neighborhood of $\partial M$, and by elliptic regularity we get that $w$ is $C^{\infty}$ in a neighborhood of $\partial M$, with $w=\partial_{\nu}w=0$. Therefore $w\equiv 0$ in this neighborhood, and by unique continuation, $w\equiv 0$ on $M\setminus\overline U$. We deduce that, on $U$, $h=-\Delta w_{|_U}$ (in $H^{-m}(U)$) and $w=\partial_{\nu}w=0$ on $\partial U$ in the appropriate sense.

\medskip

We are ready to conclude: let $v\in X$. Then
$$
T(v)=\int_U hv=-\int_U \Delta w v=\int_{\partial U}(\partial_{\nu}v w-\partial_{\nu}wv)-\int_Uw \Delta v=0.
$$
\end{proof}

We conclude this section by stating the following additional generic result, even though it is not used for the purposes of this paper.

\begin{theorem}\label{T_gen_2}
Let $M$ be a compact smooth $n$-dimensional manifold with smooth boundary $\partial M$. Then for a generic metric $g$ on $M$ all Steklov eigenfunctions have no singular zeros in $M$.
\end{theorem}
The proof of this result can be performed essentially as that of Theorem \ref{T_gen}, with $\beta$ replaced by the function
$$
\alpha:Q\times M\to \mathbb R
$$
defined by
$$
\alpha(f,\sigma,g,x)=Hf(x),
$$
and is accordingly omitted (see also \cite{U,wang}).

\begin{remark}
We can see in the case of the unit ball of $\mathbb R^n$ that the properties stated in Theorems \ref{T_gen} and \ref{T_gen_2} may fail. In fact, $0$ is a critical point for $u_k$ when $k\geq n+2$, and it is degenerate when $k\geq\frac{(n+1)(n+2)}{2}$. However the second eigenvalue is $1$, and a second eigenfunction is a coordinate function (or a linear combination of coordinate functions), which is Morse, and has no interior critical zeros. Hence it is reasonable to expect that these two properties, which we have seen to hold generically for any eigenfunction, are always satisfied by a second eigenfunction, as proved in Section \ref{S2}  for the genus $0$ case. It remains open the question whether the hypothesis on the genus can be removed.
\end{remark}

\section*{Acknowledgments}
The authors wish to thank Katie Gittins and Iosif Polterovich for useful discussions on the topic.

\bibliography{bibliography}{}
\bibliographystyle{abbrv}
\end{document}